\documentclass[12pt]{amsart}

\usepackage{amsmath}
\usepackage{amssymb}
\usepackage{amsfonts}
\usepackage{amsthm}
\usepackage{enumerate}
\usepackage{bbm}

\usepackage{esint}
\usepackage{graphicx}
\usepackage{verbatim}

\newcommand{\ls}{\lesssim}

\newcommand{\ci}[1]{_{{}_{\scriptstyle{#1}}}}

\newcommand{\Be}{\begin{equation}}
\newcommand{\Ee}{\end{equation}}

\newcommand{\Bm}{\begin{multline}}
\newcommand{\Em}{\end{multline}}
\newcommand{\Bea}{\begin{eqnarray}}
\newcommand{\Eea}{\end{eqnarray}}
\newcommand{\Beas}{\begin{eqnarray*}}
\newcommand{\Eeas}{\end{eqnarray*}}
\newcommand{\Benu}{\begin{enumerate}}
\newcommand{\Eenu}{\end{enumerate}}
\newcommand{\Bi}{\begin{itemize}}
\newcommand{\Ei}{\end{itemize}}

\def\XXint#1#2#3{{\setbox0=\hbox{$#1{#2#3}{\int}$ }
\vcenter{\hbox{$#2#3$ }}\kern-.5\wd0}}

\def\intslash{\fint}

\def\qsl{{\rlap{\kern  .32em $\mspace {.5mu}\backslash$ }\int_{Q_x}}}
\def\Re{\operatorname{Re\,}}

\def\emph#1{{\it #1 }}
\def\diam{{\text{\rm  diam}}}

\def\cf{{\it cf}}

\def\supp{{\text{\rm supp}}}

\def\inn#1#2{\langle#1,#2\rangle}

\def\noi{\noindent}

\def\lc{\lesssim}

\def\eps{\varepsilon}

\def\om{\omega}

\def\fS{{\mathfrak {S}}}

\def\bbN{{\mathbb {N}}}

\def\bbR{{\mathbb {R}}}

\def\cB{{\mathcal {B}}}

\def\cE{{\mathcal {E}}}

\def\cG{{\mathcal {G}}}

\def\cK{{\mathcal {K}}}
\def\cL{{\mathcal {L}}}
\def\cM{{\mathcal {M}}}
\def\cN{{\mathcal {N}}}

\def\cR{{\mathcal {R}}}
\def\cS{{\mathcal {S}}}
\def\cT{{\mathcal {T}}}
\def\cU{{\mathcal {U}}}

 at 10 true pt

\def\be#1{\begin{equation}\label{ #1}}
\def\endeq{\end{equation}}
\def\endal{\end{align}}
\def\bas{\begin{align*}}
\def\eas{\end{align*}}
\def\bi{\begin{itemize}}
\def\ei{\end{itemize}}

\def\eps{\varepsilon}

\def\emph#1{{\it #1}}
\def\textbf#1{{\bf #1}}

\theoremstyle{plain}
   \newtheorem{theorem}{Theorem}[section]

   \newtheorem{lemma}[theorem]{Lemma}
   \newtheorem{corollary}[theorem]{Corollary}

   \newtheorem{theorem*}{Theorem}

\theoremstyle{remark}

\theoremstyle{definition}

\renewcommand{\le}{\leqslant}
\renewcommand{\ge}{\geqslant}

\begin{document}

\title[A Calder\'on--Zygmund estimate with applications]{A Calder\'on--Zygmund estimate with applications to generalized Radon transforms and Fourier integral operators}

\author[M. Pramanik   \ \ K. Rogers   \ \  A. Seeger]
{Malabika Pramanik  \ \ Keith M. Rogers \  \ Andreas Seeger}

\address{Malabika Pramanik\\Department of Mathematics\\University of British Columbia\\Room 121, 1984 Mathematics Road\\Vancouver, B.C., Canada V6T 1Z2} \email{malabika@math.ubc.ca}

\address{Keith Rogers \\
Instituto de Ciencias Matematicas
CSIC-UAM-UC3M-UCM \\
Madrid 28049, Spain} \email{keith.rogers@icmat.es}

\address{Andreas Seeger   \\Department of Mathematics\\ University of Wisconsin-Madison\\Madison, WI 53706, USA}\email{seeger@math.wisc.edu}


\begin{thanks} {M.P. supported  in part by NSERC grant 22R82900.
K.R. supported in part by MEC grant
 MTM2007-60952. A.S. supported in part by NSF grant 0652890.}
\end{thanks}

\begin{abstract} We prove a Calder\'on--Zygmund type estimate which
can be applied to sharpen known  regularity results on spherical means,
Fourier integral operators, generalized Radon transforms and singular oscillatory integrals.
\end{abstract}
\maketitle

The main theme in this paper is to strengthen various sharp
$L^p$--Sobolev regularity results for integral operators. To
illustrate this we consider the example of spherical means.

Let $\sigma$ denote  surface measure on the unit sphere. Since
$$|\widehat \sigma(\xi)|\le C(1+|\xi|)^{-\frac{d-1}2} $$  the
convolution operator $f\mapsto f*\sigma$ maps $L^2$ to the Sobolev
space $L^2_{(d-1)/2}$. By complex interpolation with an
$L^\infty$--BMO--estimate,
 Fefferman and Stein
\cite{FS} proved that
 the
operator maps $L^p$ to $L^p_{(d-1)/p}$ for $2<p<\infty$; here the
regularity parameter $\alpha=(d-1)/p$ is optimal. It turns out, however,
that  the
$L^p$--Sobolev result can be improved within the scale of
Triebel--Lizorkin  spaces \cite{triebel} in two ways.

We recall the  definition of the quasinorm $$\|f\|_{F^p_{\alpha,q}}=
\Big\|\Big(\sum_{k=0}^\infty2^{k\alpha q}|\Pi_k
f|^q\Big)^{1/q}\Big\|_p$$ which we will use for $1<p<\infty$ and
$0<q<\infty$. Here the operators $\Pi_k$ are defined by the standard
smooth Littlewood--Paley cutoffs, so that $\widehat{\Pi_k f}$ is
supported in $\{2^{k-1}\le |\xi|\le 2^{k+1}\}$ for $k>1$ and in a
neighborhood of the origin for $k=0$; we assume that
$\sum_{k=0}^\infty \Pi_kf =f$ for all Schwartz functions. It is well
known, and immediate from Littlewood--Paley theory and embeddings
for sequence spaces, that $L^{p}\subset F^p_{0,p}\equiv B^p_{0,p}$,
$2\le p<\infty$ and, for all $p\in (1,\infty)$, $F^p_{\alpha, r}
\subset F^p_{\alpha, s} \subset F^p_{\alpha,2}= L^p_\alpha$ if
$0<r\le s\le 2$. Thus the inequalities
\Be\label{spherical}\|f*\sigma\|\ci{F^{p}_{\frac{d-1}{p},r} }\le
C_{p,r}
  \|f\|\ci{F^p_{0,p}}, \quad r>0,
\quad
  2<p<\infty,
\Ee strengthen the standard regularity result. The case  $r=1$ also
implies an $F^{p}_{0,\infty}\to F^{p}_{\alpha,p}$ estimate for
  $1<p<2$ and $\alpha=(d-1)/p'$,
by duality and composition with Bessel derivatives
$(I-\Delta)^{\alpha/2}$. Related phenomena have recently been
observed in articles on  space--time (or local smoothing) estimates
for Schr\"odinger equations \cite{rs} and wave equations \cite{hns}.

In \S\ref{czsect} we formulate a general result  which covers the
 spherical means and many other related applications. These are
 discussed
in
 \S\ref{applsect}.

\medskip

\section{A Calder\'on--Zygmund estimate}\label{czsect}
For each $k\in \bbN$, we consider operators $T_k$ defined on the
Schwartz functions $\cS(\bbR^d)$ by
$$T_kf(x)=\int K_k(x,y) f(y) dy,$$
where each $K_k$ is a continuous and bounded  kernel
(this qualitative assumption is made to
avoid measurability questions).
Let $\zeta\in \cS(\bbR^d)$. Define  $\zeta_k= 2^{kd}\zeta(2^k\cdot)$
and $$P_k f=\zeta_k*f.$$
In applications the   operators $P_k$ often arise from
  dyadic frequency decompositions, however we emphasize that no cancellation
condition on $\zeta$ is needed in the following result.

\begin{theorem}\label{czthm}
Let  $0<a<d$, $\eps>0$, and $1<q<p<\infty$. Assume the operators
$T_k$  satisfy \Be\label{Lphyp} \sup_{k>0} 2^{ka/p}\|T_k\|_{L^p\to
L^p} \le A \,. \Ee\Be\label{Lqhyp}\sup_{k>0} 2^{ka/q}\|T_k\|_{L^q\to
L^q} \le B_0\,. \Ee Furthermore let $\Gamma\ge 1$, and  assume  that
for each cube $Q$  there is a measurable set $\cE_Q$ so that
\Be\label{excset}  |\cE_Q| \le \Gamma \max\{ |Q|^{1-a/d} , |Q| \},
\Ee and for every $k\in \bbN$ and   every cube $Q$ with $2^k
\diam(Q)\ge 1$, \Be\label{Linftyhyp}
 \sup_{x\in Q}\int_{\bbR^d\setminus \cE_Q}|K_k(x,y)| \,dy \le
B_1 \max\big\{ \big(2^k\diam(Q)\big)^{-\eps}, 2^{-k\eps} \big\}.
\Ee Let \Be\label{cBdef} \cB:=
B_0^{q/p}(A\Gamma^{1/p}+B_1)^{1-q/p}.\Ee
 Then
there is a $C$ (depending only on $d,\zeta,a,\eps,p,q,r$)
 so
that \Be\label{concl}
\Big\| \Big(\sum_k 2^{k ar/p} | P_k T_k f_k|^r\Big)^{1/r} \Big\|_p
\le C
 A
\Big[\log\big(3 +\frac{\cB}{A}\big)\Big]^{1/r-1/p}
\Big(\sum_k\|f_k\|_p^p \Big)^{1/p}. \Ee
\end{theorem}

 In some interesting
applications $A\ll \cB$ so that the logarithmic growth in
\eqref{concl} is helpful. The power of the logarithm is sharp (see
\cite{limcriv}, \cite{sz}, \cite{triebel} for a relevant
 counterexample and
\cite{SeTAMS}, \cite{CaSe} for positive results on
families of  translation invariant and pseudo-differential operators).

To prove Theorem \ref{czthm} we begin with a standard
$L^\infty$--bound. In what follows the notation
$\intslash_Q f$ will be used for  the average $|Q|^{-1}\int_Q f$.

\begin{lemma} \label{Linftylemma} Assuming \eqref{Lphyp}, \eqref{excset} and
\eqref{Linftyhyp}, the following statements hold true.

(i)
If $2^{-k}\le \diam(Q) \le 1$, then
\Be\label{fixedcubesmall} \intslash_Q |P_kT_k h|\, dy \le
C\big(A\Gamma^{1/p} \big(2^{k}\diam(Q)\big)^{-a/p}+B_1
\big(2^{k}\diam (Q)\big)^{-\eps} \big) \|h\|_\infty. \Ee

(ii) If $\diam(Q)\ge 1$, then \Be \label{fixedcubelarge} \intslash_Q
|P_kT_k h|\, dy \le C \big(A\Gamma^{1/p} 2^{-ka/p}+B_1
2^{-k\eps}\big) \|h\|_\infty. \Ee
\end{lemma}

\begin{proof}
We split $h=h\chi_{\cE_Q}+h\chi_{\bbR^d\setminus\cE_Q}$. By
H\"older's inequality, \eqref{Lphyp}, and
then~\eqref{excset},\footnote{The expression $A\ls B$ denotes $A\le
CB$, where the value of the positive constant $C$ will vary from
line to line.}
\begin{align*}
\intslash_Q \big|T_k [h\chi_{\cE_Q}]\big|\, dx &\le
|Q|^{-1/p}\Big(\int\big|T_k [h\chi\ci{\cE_Q}]\big|^p\, dx\Big)^{1/p}
 \\&\lc |Q|^{-1/p}   A 2^{-ka/p}\big\|h\chi\ci{\cE_Q}\big\|_p
\lc A 2^{-ka/p}|Q|^{-1/p} |\cE_Q|^{1/p} \|h\|_\infty
\\&\lc
A \Gamma^{1/p} 2^{-ka/p} \max\{\diam(Q)^{-a/p}, 1\} \|h\|_\infty\,.
\end{align*}
On the other hand, by \eqref{Linftyhyp},
\begin{align*}
\intslash_Q \big|T_k [h\chi\ci{\bbR^d\setminus\cE_Q}]\big|\, dx &\le
\sup_{x\in Q}\int_{\bbR^d\setminus\cE_Q}  |K_k(x,y)| h(y)\, dy
\\&\lc B_1 \max \big\{\big(2^k\diam (Q)\big)^{-\eps}, 2^{-k\eps}\big\}
\,\|h\|_\infty\, .
\end{align*}
A combination of these two bounds shows that
the stated estimates hold with  $P_kT_k$ replaced  by $T_k$.

We now use straightforward estimates to incorporate the operators $P_k$. In view of the rapid decay of $\zeta$ we have
$$
\intslash_Q |P_kT_k h(x)|\, dx \le C_N \intslash_Q \int
\frac{2^{kd}}{(1+2^k|x-w|)^N} |T_k h(w)|\,dw\, dx.
$$
Now for $m=0,1,2,\dots$ we let  $Q_m^*$  denote the cube parallel to
$Q$ with the same center, but with sidelength equal to $2^{m+1}$
times the sidelength of $Q$. Then the last estimate (with $N\gg d$)
implies
\begin{multline*}\intslash_Q |P_kT_k h(x)|\, dx \\
\le C_N'
\intslash_{Q_0^*} |T_k h(w)|\, dw + \sum_{m=1}^\infty \big(2^k
\diam(Q_m^*)\big)^{d-N} \intslash_{Q_m^*}|T_k h(w)| \,dw
\end{multline*}
The term corresponding to $m=0$ has already been estimated and,
also by the bounds above applied to $Q_m^*$,
  the $m$th term is
controlled  by
\begin{multline*}
2^{-m(N-d)} \big(2^k\diam(Q)\big)^{d-N} \left(\frac{A\Gamma^{1/p}
2^{-ma/p}}{ \big(2^k\diam(Q)\big)^{a/p}} +\frac{ B
2^{-m\eps}}{\big(2^k\diam (Q)\big)^{\eps}}\right) \|h\|_\infty\,
\end{multline*}
if $2^m \diam (Q)\le 1$, and by
\begin{equation*}
2^{-m(N-d)} \big(2^k\diam(Q)\big)^{d-N} \big(A\Gamma^{1/p} 2^{-ka/p}
+ B 2^{-k\eps}\big)\|h\|_\infty
\end{equation*}
if $2^m \diam (Q)> 1$. We sum in $m$ to obtain the claimed result.
\end{proof}

\begin{proof}[\bf Proof of Theorem \ref{czthm}]
We first note that the asserted inequality for $r=p$ follows by assumption \eqref{Lphyp}  and Fubini's theorem.
We prove the theorem for $r\le 1$ and the intermediate cases $1<r<p$ follow by interpolation.

By the monotone convergence theorem it suffices to prove
\eqref{concl} for all {\it finite} sequences  $F=\{f_k\}_{k\in
\bbN}$, {\it i.e.}, we may assume that $f_k=0$ for large $k$.

We use the Fefferman--Stein theorem \cite{FS} for the $\#$--maximal
operator. The left hand side of \eqref{concl} is then rewritten and
estimated as
\begin{align*}
&\Big\|\sum_k |2^{ka/p} P_kT_kf_k|^r \Big\|_{p/r}^{1/r}\,\lc
\\&\Big\|\sup_{Q:x\in Q} \intslash_Q
\Big|\sum_k |2^{ka/p} P_kT_kf_k(y)|^r
-\intslash_Q \sum_k |2^{ka/p} P_kT_kf_k(z)|^r\,dz\Big| dy \Big
\|_{L^{p/r}(dx)}^{1/r}
\\&\lc\Big\|\sup_{Q:x\in Q} \sum_k 2^{kar/p}
\intslash_Q \intslash_Q | P_kT_kf_k(y)-P_kT_kf_k(z)|^r \, dz \,dy
\Big\|_{L^{p/r}(dx)}^{1/r}.
\end{align*} In the last step we simply use
$|u^r-v^r|\le |u-v|^r$ for nonnegative  $u,v$ and $0<r\le 1$,
combined with the triangle inequality.

Note that the application of the Fefferman-Stein inequality is valid
because of our {\it a priori} assumption involving finite sums.

Given a sequence $f_k$ we can choose cubes $Q(x)$ depending
measurably on $x$  so that the supremum in $Q$ can be up to a factor
of two realized by the choice of $Q(x)$. This means that it suffices
to prove the inequality
\begin{multline}\label{Qxchoice}
\Big\| \sum_k 2^{kar/p}
\intslash_{Q(x)} \intslash_{Q(x)}
| P_kT_kf_k(y)-P_kT_kf_k(z)|^r \, dz \,dy
\Big\|_{L^{p/r}(dx)}^{1/r}
\\\le  CA
\Big[\log\big(3 +\frac{\cB}{A}\big)\Big]^{1/r- 1/p}
\Big(\sum_k\|f_k\|_p^p \Big)^{1/p}
\end{multline}
where  $C$ does not depend on the choice of $x\mapsto Q(x)$.
We define $L(x)$ to be the integer $L$ for which the sidelength of
$Q(x)$ belongs to $[2^L,2^{L+1})$.

Let $X=\{x: L(x) \le 0\}$. We shall first
estimate  the $L^{p/r}$ norm  over $X$ (the main and more
interesting part) and then provide  the bound on $L^{p/r}(\bbR^d\setminus X)$
separately.

Define  $$\cG_k h(x) = \Big(\intslash_{Q(x)} \intslash_{Q(x)} |
P_kT_kh(y)-P_kT_kh(z)|^r \, dz \,dy \Big)^{1/r}$$ so that the left
hand side of \eqref{Qxchoice}  is equal to $\|\sum_k 2^{ka r/p}
|\cG_k f_k|^r\|_{p/r}^{1/r}\,.$ Let $\cN$ be a positive integer (it
will later be chosen as $C\log\big(3 +\frac{\cB}{A}\big)$ with a
large  $C$). For $x\in X$ we  split the $k$--sum into three pieces
acting on $F=\{f_k\}$,
$$
\sum_k 2^{kar/p} |\cG_k f_k(x)|^r
=|\fS^{\text{low}}[F](x)|^r+|\fS^{\text{mid}}[F](x)|^r
+|\fS^{\text{high}}[F](x)|^r
$$
where
\begin{align*}
\fS^{\text{low}}[F](x) &= \Big(\sum_{k+L(x)<0} 2^{k ar/p} |\cG_k
f_k(x)|^r \Big)^{1/r}
\\
\fS^{\text{mid}}[F](x) &=
\Big(\sum_{0\le k+L(x)\le \cN} 2^{k ar/p}
|\cG_k f_k(x)|^r\Big)^{1/r}
\\
\fS^{\text{high}}[F](x) &= \Big( \sum_{k+L(x)>\cN} 2^{k ar /p}
|\cG_k f_k(x)|^r\Big)^{1/r}.
\end{align*}
We need to bound the $L^p$ norms of the three terms by the right
hand side of \eqref{Qxchoice}. The terms
 $\fS^{\text{low}}[F]$ and
 $\fS^{\text{mid}}[F]$ will  be estimated  by using just hypothesis
\eqref{Lphyp}.

To bound  $\fS^{\text{low}}[F]$ we first consider
the expression
\begin{multline*}P_kT_k f_k(y)-P_kT_k f_k(z) =\\
\int_0^1\int \inn {2^k(y-z)}{2^{kd}\nabla \zeta (2^k(z-w+s(y-z))} T_kf_k(w) dw\, ds.
\end{multline*}
For $y,z\in Q(x)$ we have $2^k|y-z|\lc 2^{k+L(x)}$, and by
H\"older's inequality and the  rapid decay of $\zeta$,
\begin{align*}
|\cG_k f_k(x)|&\le
\Big(\intslash_{Q(x)} \intslash_{Q(x)}
| P_kT_kf_k(y)-P_kT_kf_k(z)|^r \, dz \,dy\Big)^{1/r}
\\&\lc  2^{k+L(x)}
M_{HL}[ T_kf_k](x).
\end{align*}
Here $M_{HL}$ denotes the standard Hardy--Littlewood maximal
operator. Now, by H\"older's inequality with respect to the
$k$--summation,
$$\Big(\sum_{k+L(x)\le 0} |2^{ka/p}\cG_k f_k(x)|^r\Big)^{1/r}\lc
\Big(\sum_k | 2^{k a/p} M_{HL}[ T_kf_k](x)|^p\Big)^{1/p}.$$
Thus 
\begin{align}
\|\fS^{\text{low}}[F]\|_p
&\le \Big(\sum_k 2^{ka} \|M_{HL}[ T_kf_k]\|_p^p \Big)^{1/p}\notag
\\ \label{lowest}
&\lc \Big(\sum_k 2^{ka} \|T_kf_k\|_p^p \Big)^{1/p}
\lc A \Big(\sum_k\|f_k\|_p^p\Big)^{1/p}\,.
\end{align}

Next we  take care of $\fS^{\text{mid}}[F](x)$ which may
 often  be considered the
main term but is also  estimated using just \eqref{Lphyp}.
Now
$$|\cG_k f_k(x)|^r\le 2\intslash_{Q(x)} |P_k T_k f_k(y)|^r dy$$
and therefore
\begin{multline*}\sum_{0\le k+L(x)\le \cN} 2^{k ar/p}|\cG_k f_k(x)|^r
\\ \lc
\intslash_{Q(x)} \cN^{1-r/p}  \Big(\sum_{0\le k+L(x)\le \cN}
|2^{ka/p}P_k T_k f_k(y)|^p\Big)^{r/p} dy.\end{multline*} By
H\"older's inequality, this implies
$$|\fS^{\text{mid}}[F](x)|
\lc \cN^{1/r-1/p} M_{HL}\big[ \big(\sum_k |2^{ka/p}P_k T_k
f_k|^p\big)^{1/p}\big](x),$$ so that
\begin{align}
\|\fS^{\text{mid}}[F]\|_p &\lc \cN^{1/r-1/p}  \Big\|
 M_{HL}\big[ \big(\sum_k
|2^{ka/p}P_k T_k f_k|^p\big)^{1/p}\big]\Big\|_p \notag
\\
&\lc \cN^{1/r-1/p}
  \big(\sum_k 2^{ka}\|P_k T_k f_k\|_p^p\big)^{1/p}\,
\notag
\\
&\lc A\, \cN^{1/r-1/p}
  \big(\sum_k\|f_k\|_p^p\big)^{1/p}.
\label{middleest}
\end{align}

We now turn to the expression $\fS^{\text{high}}$ which we estimate
for $L(x)\le 0$. Again by H\"older's inequality,
\begin{align*}
\fS^{\text{high}}[F](x)&\le \Big(2\sum_{k> \cN-L(x)} 2^{kar/p}
\intslash_{Q(x)}|P_kT_k f_k(y)|^r dy\Big)^{1/r}
\\
&\le\Big(2\sum_{k> \cN-L(x)} 2^{kar/p}\Big( \intslash_{Q(x)}|P_kT_k
f_k(y)|\, dy\Big)^r \Big)^{1/r}\,.
\end{align*}
If $r<1$ then we  choose a small
$\delta>0$
and use H\"older's inequality with respect to the $k$ summation to get
\begin{align}\label{withoutr}
\fS^{\text{high}}[F](x)&\le  C(r,\delta) \sum_{k> \cN-L(x)}
2^{ka/p}2^{(k+L(x))\delta} \intslash_{Q(x)}|P_kT_k f_k(y)|\, dy
\end{align}
where
$$C(r,\delta)= 2^{1/r}
\Big(\sum_{k> \cN-L(x)} 2^{-(k+L(x)) \delta
\frac{r}{1-r}}\Big)^{1-r} \lc 2^{-\cN \delta r}(r\delta)^{r-1},$$ so
that $C(r,\delta)\lc (r\delta)^{r-1} $.

In order to estimate the expression \eqref{withoutr} it suffices to
bound the $L^p$--norm of
$$ \cT^{\text{lin}}[F](x) =
\sum_{k> \cN-L(x)} 2^{ka/p}2^{(k+L(x))\delta}
\intslash_{Q(x)}\om_k(x,y) P_kT_k f_k(y)\, dy $$ where $\om_k(x,y)$
are measurable functions satisfying $\sup_{x,y,k}|\om_k(x,y)|\le 1$,
with the constants in the estimates independent of the particular
choice of the $\om_k$. We now fix one such choice.

Write $n=k+L(x)$, so that $n> \cN$, and define, for $0\le \Re(z)\le
1$, \Be\label{Sndef} S^z_nF(x) =2^{(n-L(x))a(1-z)/q}
\intslash_{Q(x)}\om_{n-L(x)}(x,y) P_{n-L(x)} T_{n-L(x)}
f_{n-L(x)}(y)\, dy. \Ee Observe that \Be\label{Tlintheta}
\cT^{\text{lin}}[F](x) = \sum_{n>\cN} 2^{n\delta}
S^\theta_nF(x)\quad   \text{ for }\ \theta= 1-\frac qp\,. \Ee We
estimate the $L^p$ norm of $S^z_n F$  for $z=\theta$ by
interpolating between an $L^q$ bound for $\Re(z)=0$ and an
$L^\infty$ bound for $\Re(z)=1$.

For $z=i\tau$, $\tau\in \bbR$ we obtain
\begin{align*}
|S^{i\tau}_n F(x)|&\le \intslash_{Q(x)}\sup_k
2^{ka/q}|P_kT_kf_k(y)|\,dy
\\
&\le M_{HL}\big[
(\sum_k|2^{ka/q}P_kT_kf_k|^q\big)^{1/q}\big](x)
\end{align*}
and therefore, by the $L^q$ estimate for $M_{HL}$, Fubini, and assumption
\eqref{Lqhyp},
$$
\|S^{i\tau}_n F\|_q\lc
\Big(\sum_k2^{ka}\|P_kT_kf_k\|_q^q\Big)^{1/q} \lc B_0
\Big(\sum_k\|f_k\|_q^q\Big)^{1/q}.
$$
The $L^\infty$ estimate for $\Re(z)=1$ follows from Lemma
\ref{Linftylemma}; for $L(x)\le 0$, we get
\begin{align*}
|S_n^{1+i\tau} F(x)|&\le \intslash_{Q(x)} |P_{n-L(x)}  T_{n-L(x)}
 f_{n-L(x)} (y)|\,dy
\\
&\lc (A\Gamma^{1/p}2^{-na/p}+ B_1 2^{-n\eps}) \|f_{n-L(x)}\|_\infty
\end{align*}
and of course $\|f_{n-L(x)}\|_\infty\le \sup_k\|f_k\|_\infty$.
Interpolating the two bounds yields, \Be\label{Snthetaest}
\|S_n^{\theta} F\|_{L^p(X)}
 \lc 2^{-\eps_0 n(1-q/p)}
\cB \Big(\sum_k\|f_k\|^p_p\Big)^{1/p} \Ee with
$\eps_0:=\min\{a/p,\eps\}$ and $\cB$ as in \eqref{cBdef}. Choosing
$\delta=(1-q/p)\eps_0/2$, this yields
\begin{align*}\|\cT^{\text{lin}}[F]\|_{L^p(X)} &\lc \sum_{n>\cN} 2^{n\delta}\|
S^\theta_n F\|_{L^p(X)}
\\
&\lc \eps_0^{-1} (1-q/p)^{-1} \cB 2^{-\cN(1-q/p)\eps_0 /2}
\Big(\sum_k\|f_k\|^p_p\Big)^{1/p}
\end{align*}
and then, by suitably choosing $\omega_k$,
\begin{align*}\|\fS^{\text{high}} F\|_{L^p(X)} \lc \eps_0^{-2} (1-q/p)^{-2} \cB 2^{-\cN(1-q/p)\eps_0 /2}
\Big(\sum_k\|f_k\|^p_p\Big)^{1/p}.
\end{align*}

We combine the three bounds for $\fS^{\text{high}}$,
$\fS^{\text{mid}}$ and $\fS^{\text{low}}$ and get
\begin{multline*}
\Big\|\sum_k 2^{kar/p}|\cG_kf_k|^r\Big\|_{L^{p/r}(X)}^{1/r}
\\\le C_r \big(A \cN^{1/r-1/p}+ \eps_0^{-2}(1-q/p)^{-2} \cB 2^{-\cN(1-q/p)\eps_0 /2}\big)
\Big(\sum_k\|f_k\|^p_p\Big)^{1/p}
\end{multline*}
and choosing $\cN= C_{\text{large}}
\log(3+\cB/A)$ (with $C_{\text{large}}$ depending on $p$ and $q$) we obtain the bound \Be \Big\|\sum_k
2^{kar/p}|\cG_kf_k|^r\Big\|_{L^{p/r}(X)}^{1/r} \le C
A\Big[\log\big(3 +\frac{\cB}{A}\big)\Big]^{1/r-
1/p}\Big(\sum_k\|f_k\|^p_p\Big)^{1/p}. \Ee
\medskip

It remains to give the
estimation on
$\bbR^d\setminus X$ (the set where $L(x)>0$) which is similar in
spirit, but
more straightforward.
We first single out the terms for $k\le \cN$ and by an estimate
similar to the one for $\fS^{\text{mid}}$ above we get
\Be\label{primero}
\Big\|\sum_{k\le \cN} 2^{kar/p}|\cG_kf_k|^r\Big\|_{L^{p/r}}^{1/r}
\lc A \cN^{1/r-1/p} \Big(\sum_k\|f_k\|^p_p\Big)^{1/p}.\Ee
On the other hand, by assumption \eqref{Lqhyp}
$$ 2^{ka/q}\|\cG_kf_k\|_q \lc B_0 \|f_k\|_q$$
and by  \eqref{fixedcubelarge}
$$ \|\cG_kf_k\|_{L^\infty(\bbR^d\setminus X)} \lc (A\Gamma^{1/p}
2^{-ka/p}+B_1 2^{-k\eps})  \|f_k\|_\infty.$$
Thus, with $\eps_0=\min\{a/p,\eps\}$ we get  by interpolation,
$$2^{ka/p}\|\cG_kf_k\|_{L^p(\bbR^d\setminus X)} \lc 2^{-k\eps_0(1-q/p)} \cB
\|f_k\|_p.$$ By a straightforward application of H\"older's
inequality,
\begin{equation}\label{secondo}\Big\|\sum_{k>\cN} 2^{kar/p}|\cG_kf_k|^r\Big\|_{L^{p/r}}^{1/r}
\lc  \eps_0^{-1/r}(1-q/p)^{-1/r}2^{-\cN\eps_0(1-q/p)/2} \cB
\sup_k\|f_k\|_p,\end{equation}
which is slightly better than the $\ell^p(L^p)$ bound that we are
aiming for. Combining \eqref{primero} and \eqref{secondo}, choosing
$\cN$ as before, yields
$$\Big\|\sum_{k} 2^{kar/p}|\cG_kf_k|^r\Big\|_{L^{p/r}(\bbR^d\setminus X)}^{1/r}
\le A\Big[\log\big(3
+\frac{\cB}{A}\big)\Big]^{1/r-1/p}\Big(\sum_k\|f_k\|^p_p\Big)^{1/p}
$$
which concludes the proof.
\end{proof}

\section{Applications}\label{applsect}

\subsection*{Integrals over hypersurfaces}
Consider the example of spherical means. For $k\in\mathbb{N}$, let
$P_k$ be a Littlewood--Paley cutoff operator $\widetilde \Pi_k$
(localizing to frequencies of size $\approx 2^k$ as in the
introduction) such that $\widetilde \Pi_k\Pi_k=\Pi_k$. Take $T_k f =
\sigma*\widetilde \Pi_kf$ and $f_k=\Pi_k f$.
If $Q$
is a cube satisfying $2^{-k}\le \diam(Q)\le 1$, with center $x_Q$,
then the exceptional set $\cE_Q$ is the tubular neighborhood of the
unit sphere centered at $x_Q$, with width $C \diam (Q)$; if $\diam
(Q)>1$ we can simply choose the double cube.
 Then the
hypotheses of Theorem \ref{czthm} are easily verified with $a=d-1$,
$q=2$, any $p>2$, and with $A$, $B_0$, $B_1$, $\Gamma$ all
comparable. Then \eqref{spherical} is implied by Theorem
\ref{czthm}.

One can extend  this observation to more general averaging operators
over hypersurfaces  which are not necessarily translation invariant.
Let $\chi\in C^\infty_c(\bbR^d\setminus \{0\})$ and let
$(x,y)\mapsto \Phi(x,y)$ be a smooth function defined in a
neighborhood of $\supp\, \chi$ and assume that
$\nabla_x\Phi(x,y)\neq 0$ and $\nabla_y\Phi(x,y)\neq 0$. Let
$\delta$ be the Dirac measure on  the real line and define the
generalized Radon--transform $\cR$ as the integral operator with
Schwartz kernel $$K_{\cR} (x,y)= \chi(x,y)\delta(\Phi(x,y)).$$ As
shown in  \cite{SoSt} (\cf.  also \cite{H}),
regularity properties of $\cR$ are determined by
the rotational curvature
$$\kappa(x,y)= \det \begin{pmatrix} \Phi_{xy} &\Phi_x\\
\Phi_y&0\end{pmatrix}\,.$$
Strengthening the results in \cite{SoSt} slightly we obtain
\begin{corollary} \label{radonappl}
(i) Suppose that $\kappa(x,y)\neq 0$ on $\supp (\chi)$. Then
$\cR$ maps $F^p_{0,p}(\bbR^d) \to F^{p}_{\frac{d-1}p,r}(\bbR^d)$, for $2<p<\infty$,  $r>0$.

(ii)  Suppose that $\kappa(x,y)\neq 0$ vanishes only of finite order
  on $\supp \chi$, i.e. there is $n$ such that $\sum_{|\gamma|\le
  n}|\partial^\gamma_{y}\kappa(x,y)|\neq 0$, then there is
  $p_0(n,d)<\infty$ so that
$\cR$ maps $F^p_{0,p}(\bbR^d) \to F^{p}_{\frac{d-1}p,r}(\bbR^d)$, for
$p_0(n,d)<p<\infty$,  $r>0$.
\end{corollary}

The proof of (i) is essentially the same as for the spherical means. One
 decomposes $\cR=\sum_{k=0}^\infty \cR_k$
where for $k>0$ the Schwartz kernel of $\cR_k$ is given by
\Be\label{Rk}R_k(x,y)= \int \eta(2^{-k}|\tau|) \chi(x,y)
\,e^{i\tau\Phi(x,y)} d\tau \Ee with a suitable $\eta$ supported in
$(1/2,2)$. One may then write $$\cR=\sum_{k=0}^\infty \Pi_k\cR_k
\Pi_k +\sum_{k=0}^\infty E_k,$$
where $E_k$ is
negligible, {\it i.e.} mapping $L^p$ to any Sobolev space $L^p_N$
with norm $\le C_N 2^{-kN}$; this decomposition follows by an
integration by parts argument in \cite{H},  and uses only the
assumptions $\Phi_x\neq 0$ and $\Phi_y\neq 0$ (see also \S2 in
\cite{SeDuke93}
 for an exposition of this kind of argument).
To estimate the main operator $\sum_{k=1}^\infty \Pi_k\cR_k \Pi_k$
we use Theorem~\ref{czthm}, setting $P_k=\Pi_k$, $f_k=\Pi_kf$,
$T_k=\cR_k$,
and choose all parameters as in the example for the spherical means.
For the exceptional sets $\cE_Q$ we choose a tubular neighborhood of
width $C\diam(Q)$ of the surface $\{y: \Phi(x_Q,y)=0\}$.

For part (ii) one decomposes the operators $\cR$ according to the
size of $\kappa$, using a suitable cutoff function of the form
$\beta_1(2^\ell|\kappa(x,y)|)$ where $\beta_1$ is supported in
$(1/2,2)$. Let  $R_k^\ell$
be
defined as in \eqref{Rk} but with $\chi(x,y)$ replaced by
$\chi(x,y)\beta_1(2^\ell|\kappa(x,y)|)$. Then the proof of
Proposition 2.2 in \cite{SoSt} shows that the operators
 $\cR_k^\ell$ are bounded on $L^2$ with operator norm
$\lc 2^{\ell M}2^{-k \frac{d-1}{2}}$
(in fact with $M=5d/2+(d-1)/2$). By the finite type assumption on $\kappa$
(and a standard  sublevel set estimate related to van der Corput's lemma)
the operator  $\cR_k^\ell$  is  bounded  on $L^\infty $ with operator norm
$\lc 2^{-\ell/n}$. Hence for $p>q>(2Mn+1)$
hypothesis \eqref{Lphyp} and \eqref{Lqhyp} are  satisfied with
$A=2^{-\ell\eps(p)}$,
$B_0=2^{-\ell\eps(q)}$
for some $\eps(p)>0$, $\eps(q)>0$. We choose the exceptional set as in
part (i) and \eqref{excset}, $\eqref{Linftyhyp}$ hold as well with some
$B_1$, $\Gamma$  independent of $\ell$.

\subsection*{Fourier integral operators}
Another application concerns general  Fourier integral operators
associated to a canonical graph. Let $\chi\in C^\infty_0(\bbR^d)$,
let $a$ be a standard smooth symbol supported in $\{\xi:|\xi|\ge
1\}$. Let $$Sf (x) = \chi(x) \int a(x,\xi) \widehat f(\xi)
\,e^{i\phi(x,\xi)} d\xi$$ where $\phi$ is smooth in
$\bbR^{d}\setminus \{0\}$ and
  $\xi\mapsto \phi(x,\xi)$ is homogeneous of degree $1$. We assume
that  $\det \phi_{x\xi}''\neq 0$ on the support of the symbol. The
following statement sharpens the $L^p$ estimates of \cite{peral},
\cite{mi1}
for the wave equation and of \cite{sss} for more general Fourier integral
operators.
 One can use
general facts about Fourier integral operators \cite{H} to see that
it implies part (i) of Corollary~\ref{radonappl}.

\begin{corollary} \label{fioappl}
Let $d\ge 2$, $2<p<\infty$, $r>0$,
  and let  $a$ be  a standard symbol of order $-(d-1)(\frac
12-\frac 1p)$. Then   $S: F^p_{0,p}(\bbR^d) \to F^{p}_{0,r}(\bbR^d)$ is
bounded.
\end{corollary}

The statement is equivalent with the $F^p_{0,p} \to
F^{p}_{(d-1)/p,r}$ boundedness of a similar  Fourier integral
operator $T$ of order $-(d-1)/2$. We use the dyadic decomposition in
$\xi$ to split  $T=T_0+\sum_{k=1}^\infty T_k$ where $T_0$ is
smoothing to arbitrary order. Exceptional sets are also constructed
as in \cite{sss}. Given a cube $Q$ with center $x_Q$ and diameter
$d_Q\le 1$ one  chooses a maximal $\sqrt{d_Q}$--separated set of
unit vectors $\xi_\nu$, thus this set has  cardinality
$O(d_Q^{-(d-1)/2})$. For each $\nu$ consider the rectangle let
$\pi_\nu$ be the orthogonal projection to the hyperplane
perpendicular to $\xi_\nu$. Form for large $C$ the rectangle
$\rho_\nu(Q)$ consisting of $y$ for which
$|\inn{y-\phi_{\xi}(x_Q,\xi_\nu)}{\xi_\nu}|\le Cd_Q$ and
$|\pi_\nu(y-\phi_{\xi}(x_Q,\xi_\nu))|\le Cd_Q^{1/2}$. The
exceptional set $\cE_Q$ for $|Q|<1$ is then defined
 to be the union of the $\rho_\nu(Q)$ and has measure $O(|Q|^{-1/d})$.
We refer to \cite{sss} for the arguments proving
$\|T_k\|_{L^p\to L^p}\lc 2^{-k(d-1)/p}$, $2<p<\infty$   and the integration by
parts arguments leading to \eqref{Linftyhyp}.

\subsection*{ Strongly singular integrals }
Define  the convolution operator $S^{b,\gamma}$ by
 $$\widehat {S^{b,\gamma}f}(\xi)=
\frac{\exp(i|\xi|^\gamma)}{(1+|\xi|^2)^{b/2}}\widehat
f(\xi)$$
We assume
$0<\gamma<1$ and $1<p<\infty$. The classical result
\cite{FS}
states that $S^{b,\gamma}$ is bounded on $L^p(\bbR^d)$ if
and only if $b\ge \gamma d|1/2-1/p|$.

Theorem \ref{czthm} can be used to upgrade the endpoint version to
\begin{corollary} For
$2<p<\infty$,
\label{strsing}$$S^{b,\gamma}:F^p_{0,p}\to F^p_{0,r}, \qquad b=b(\gamma)= \gamma d(1/2-1/p), \quad
r>0.$$
\end{corollary}

To prove it we define  $\widehat {T^\gamma
f}(\xi)=(1+|\xi|^2)^{-\frac{\gamma d}{2p}}\widehat
{S^{b(\gamma),\gamma}f}(\xi)$. For $\diam(Q)<1$ we choose for the
exceptional set $\cE_Q$ the cube with the same center but diameter
$C (\diam (Q))^{1-\gamma}$, for large $C$. Then the verification of
the hypotheses with $a=\gamma d$ is done using the arguments in
\cite{FS} or \cite{mi2}.


\medskip

\noi{\it Remarks.}
(i) For the range
 $2\le p<s$, it is known that
the operator $S^{b(\gamma),\gamma}$  is not  bounded on $F^p_{0,s}$
(see~\cite{chse}).

(ii)
There are also corresponding results for the range $\gamma>1$
which improve on  the results in \cite{mi2},
but they do not fit precisely our setup of Theorem \ref{czthm}
(\cf. \cite {rs} for the corresponding  smoothing space time estimate).

\subsection*{Integrals over curves}

We consider the generalized Radon transform associated to curves
given by the equations $\Phi_i(x,y)=0$, $i=1,\dots, d-1$, where
the $\nabla_x\Phi_i$ are linearly  independent and
the $\nabla_y\Phi_i$ are linearly  independent, for
$(x,y)$ in a neighborhood $\cU=X\times Y$ of the  support of
a $C^\infty_c$ function $\chi$.
For simplicity (and without loss of generality) we assume that
$\Phi_i(x,y):=S^i(x,y_d)-y_i$ for $i=1,\dots,d-1$, and $\nabla_x S^i$ are linearly independent.

An important model  case arises when $S^i(x,y_d)=
x_i+(x_d-y_d)^{d+1-i}$ ({\it i.e.} for convolution with arclength
measure on the curve $(t^d,t^{d-1},\dots,t)$, for a compact
$t$-interval). The complete sharp $L^p$--Sobolev estimates
for $2<p<\infty$ are unknown in dimension $d\ge 3$. However in three
dimensions the sharp estimates are known for some range of
large $p$ (see \cite{pr-se}),
and this result is strongly related to deep
 questions on  Wolff's inequality for decompositions
 of cone multipliers \cite{Wolff1}.
%
A variable coefficient generalization of the result in \cite{pr-se}
is in \cite{pr-se3}. To discuss and apply the latter result we now
let $\delta$ be the Dirac measure on  $\bbR^{d-1}$ and define the
generalized Radon transform
$\cR$
as the operator with Schwartz kernel
$$\cK (x,y)= \chi(x,y)\delta(\vec\Phi(x,y)).$$ Again we shall also
consider the dyadic pieces $\cR_k$  with Schwartz kernel
\Be\label{Rkcurves}R_k(x,y)= \int \beta(2^{-k}|\tau|) \chi(x,y)
e^{i\tau\cdot\vec\Phi(x,y)} d\tau \Ee The analogue of the rotational
curvature now depends on $\tau$; we define it as a homogeneous of
degree zero function and, for $|\tau|=1$, set
$$\kappa(x,y,\tau)= \det \begin{pmatrix}\tau\cdot\vec \Phi_{xy} &\vec \Phi_x\\
\vec\Phi_y&0\end{pmatrix}=\sum_{i=1}^{d-1} \tau_i \det\begin{pmatrix}
S^i_{xy_d}&S^1_x&\cdots&S^{d-1}_x \end{pmatrix}.$$
Note that for $d\ge 3$ there are always directions where
$\kappa(x,y,\tau)$ vanishes.

In \cite{pr-se3} we consider the case $d=3$ and refer to this paper
for further discussion. Let $\cM=\{(x,y)\in\cU:
\vec\Phi(x,y)=0\}$
and let
$N^*\cM$ be the conormal bundle. We assume that $(N^*\cM)'$ is a
folding
canonical relation  and satisfies  an additional  curvature
condition.
To
describe the latter one  consider the fold surface
 $$\cL = \{(x, \tau\cdot\vec\Phi_x(x,y), y,
-\tau\cdot\vec\Phi_y(x,y)) : \vec\Phi(x,y)=0,\,\,
\kappa(x,y,\tau)=0\},$$
and assume that the
projection $\cL\to X$  has surjective differential. Thus for any
fixed~$x$ the set $\Sigma_x= \{\xi\in \bbR^3 : (x,\xi,y,\eta)\in
\mathcal L \text{ for some $(y,\eta)$}\}$ is a two-dimensional conic
hypersurface, and  the additional curvature assumption is  that
$\Sigma_x$  has one nonvanishing principal curvature everywhere
(see \cite{gs}, \cite{pr-se3} for further discussion). For
$d=3$ this covers perturbation of the translation invariant model
case.

Fix $\ell$ and, for $k> 3\ell$, define
\begin{align*}R_k^\ell(x,y)= \int \eta(2^{-k}|\tau|) \chi(x,y)
\widetilde \beta_1 (2^\ell (\kappa(x,y,\tfrac{\tau}{|\tau|}))
\,e^{i\tau\cdot\vec\Phi(x,y)} d\tau
\end{align*}
where $\widetilde \beta_1$ is supported in $\{\xi: C^{-1}\le
|\xi|\le C\}$ for large $C$, and, for $k=3\ell$, define
$R_k^\ell(x,y)$ in the same way but with $\beta_1$
replaced
 by
$\beta_0$, a smooth cutoff function which is equal to $1$ in a
$C$-neighborhood of the origin. Let $\cR_k^\ell $ be the operator
with Schwartz kernel $R^\ell_k$. We then have to estimate the $L^p$
operator norm for
$$\cR^\ell:=\sum_{k\ge 3\ell}\cR^{\ell}_k,$$ for each $\ell>0$.

In \cite{pr-se3} it is  shown, based on the
previously
mentioned Wolff inequality, that under the above assumptions
$$\big\|\cR^{\ell}_k\big\|_{L^p\to L^p} \lc C(\epsilon_\circ, p) 2^{-k/p}
2^{-\ell(1-\epsilon_\circ)/p}, \qquad p>p_W. $$
Here $(p_W,\infty)$ is the range of Wolff's inequality
(in \cite{Wolff1}  $p_W=74$, but this has been improved since).
Standard $L^2$ estimates (see \cite{PhSt},
\cite{phongsurvey}) give that for $k\ge 3\ell$ the operators
$\cR_k^\ell$ are bounded on $L^2$ with norm
$O(2^{(\ell-k)/2})$.
By interpolation,
$$\big\|\cR^{\ell}_k\big\|_{L^p\to L^p} \lc 2^{-\frac kp}
2^{-\ell\epsilon(p)} \text{ with } \epsilon(p)>0 \text{ for }
 p>(p_W+2)/2.
$$
We claim that this yields  the boundedness result
\begin{corollary}\label{curveresult}
The operator $\cR$
maps $F^p_{0,p}(\bbR^3)$ to $F^p_{\frac1p,r} (\bbR^3)$, for
 $\frac{p_W+2}{2}<p<\infty$.
\end{corollary}
To see this we use the assumption that $\nabla_xS^i$ are linearly independent and thus by integration by parts one can find a constant $C_0$
depending on $\vec S$ so that
\begin{multline*}\|\Pi_k \cR^\ell_{k'} \Pi_{k''}\|_{L^p\to L^p}
\le C_N \min\{ 2^{-kN},  2^{-k'N},  2^{-k''N}\} \\ \text{ provided that }
\max\{|k-k'|, |k'-k''|\}\ge C_0, \quad k'\ge 3\ell.
\end{multline*}
Straightforward arguments (such as
 those
used for the error terms in the proof of Corollary \ref{radonappl})
reduce matters to the inequality \Be\label{essentialineq}
\Big\|\Big(\sum_{k>0} |2^{k/p}\Pi_{k+s_1} \cR^\ell_k
\Pi_{k+s_2}f|^r\Big)^{1/r}\Big\|_p \lc 2^{-\ell\epsilon'(p)}
\Big\|\Big(\sum_{k>0} |\Pi_{k+s_2}f|^p\Big)^{1/p}\Big\|_p\quad \Ee
with $\epsilon'(p)>0$ for $p>(p_W+2)/2$. Here $|s_1|\le C_0$ and
$|s_2|\le C_0$. Indeed we apply, for fixed $\ell$, Theorem
\ref{czthm} with $P_k=\Pi_{k+s_1}$, $f_k=\Pi_{k+s_2}f$, and
$T_k=\cR_k^\ell$ if $k\ge 3\ell$ (and $T_k=0$ otherwise). For
$p>q>(p_W+2)/2$ assumption \eqref{Lphyp} holds with $A\lc
2^{-\ell\epsilon(p)}$ and assumption \eqref{Lqhyp} holds with
$B_1\lc 2^{-\ell\epsilon(q)}$. We check assumption
\eqref{Linftyhyp}. By an integration by parts argument we derive the
crude bound
$$|R_k^\ell(x,y)|\le C_N \frac{2^{2k}}{(1+2^{k-\ell} |y'-\vec S(x_Q,y_3)|
)^N}\,.
$$
Now for a
given cube $Q$ with center
$x_Q$ we let
$$\cE_Q:=\{y: |y'-\vec S(x_Q,y_3)|\le C 2^\ell\diam(Q)\}$$
if $\diam(Q)\le 1$. If $\diam (Q)\ge 1$ then we let $\cE_Q$ be a ball
 of diameter $C2^\ell\diam(Q)$ centered at $x_Q$. Clearly
assumptions \eqref{excset} and \eqref{Linftyhyp} are  satisfied
with $\Gamma \lc 2^{3\ell}$ and
$B_1\lc 2^{2\ell}$. By Theorem  \ref{czthm}
$$\Big\|\Big(\sum_{k\ge 3\ell}|2^{k/p}P_k\cR^\ell_k  f_k|^r\Big)^{1/r}\Big\|_p
\lc (1+\ell)2^{-\epsilon'(p)\ell} \Big(\sum_k\|f_k\|_p^p\Big)^{1/p},
\,\,\, p>\frac{p_W+2}2,
$$
which concludes the proof of \eqref{essentialineq} and yields
$$\big\|\cR^\ell f\big\|\ci{F^p_{\frac{1}{p},r}}
\lc (1+\ell)2^{-\eps(p)\ell} \|f\|\ci{F^p_{0,p}}.
$$
Corollary \ref{curveresult} follows by summation in $\ell\ge 0.$

\medskip \noi{\it Remark.}
A similar strengthening, with a similar argument,
 applies to the restricted X--ray transform model in \cite{pr-se2}.

\end{document}